\documentclass{amsart}
\usepackage{latexsym,amsmath,amsthm,amssymb}

\usepackage{color}

\usepackage{graphicx}

\newtheorem{theorem}{Theorem}[section]

\newtheorem{corollary}[theorem]{Corollary}

\theoremstyle{remark}

\newtheorem{remark}[theorem]{Remark}

\theoremstyle{definition}
\newtheorem{definition}[theorem]{Definition}

\newtheorem{problem}[theorem]{Problem}

\begin{document}

\title{Exchangeable optimal transportation and log-concavity}

\author{Alexander~V.~Kolesnikov}
\address{ Higher School of Economics, Moscow,  Russia}
\email{Sascha77@mail.ru}

\vspace{5mm}
\author{Danila~A.~Zaev}
\address{ Higher School of Economics, Moscow,  Russia}
\email{zaev.da@gmail.com}

\thanks{ 
The first named author was supported by RFBR project 14-01-00237 and  the DFG project  CRC 701.
This study (research grant No 14-01-0056) was supported by The National Research University-Higher School of Economics' Academic Fund Program in 2014/2015.
}


\begin{abstract}
We study the Monge and Kantorovich transportation problems on $\mathbb{R}^{\infty}$ within the class of exchangeable measures.
 With the help of the de Finetti decomposition theorem the problem is reduced
 to an unconstrained optimal transportation problem on the Hilbert space.
We find sufficient conditions for convergence of finite-dimensional approximations 
to the Monge solution.
The result holds, in particular, under certain analytical assumptions involving
log-concavity of the target measure. 
 As a by-product we obtain the following result: any uniformly log-concave
 exchangeable sequence of random variables is i.i.d.
\end{abstract}

\maketitle

\section{Introduction}

We consider  the Polish linear space $\mathbb{R}^{\infty}$ equipped
with the standard Borel sigma-algebra and two Borel exchangeable probability measures $\mu$ (source measure) and $\nu$ (target measure).
A Borel measure  is called exchangeable if it is invariant 
with respect to any permutation of  finite number of coordinates, i.e.
under any  linear operator $g$ satisfying
$$
g(x_1, x_2, \ldots, x_n, x_{n+1}, \ldots ) = (x_{\sigma(1)}, x_{\sigma(2)}, \ldots, x_{\sigma(n)}, x_{n+1}, \ldots ),
$$
for every point
$
x = (x_1, x_2, \ldots ) \in \mathbb{R}^{\infty}
$
and some $\sigma \in S_n$ (permutation group of $n$ elements).
We denote the union of all such $g$  by $\mathcal{S}_{\infty}$
and call it the infinite permutation group.

We say that a measure $\pi$ on $X \times Y$, where $X=Y=\mathbb{R}^{\infty}$,
is exchangeable if it is invariant with respect to any mapping
$$
(x,y) \mapsto (g(x), g(y)), \ g \in \mathcal{S}_{\infty}.
$$
Finally, a mapping $T \colon \mathbb{R}^{\infty} \mapsto \mathbb{R}^{\infty}$
is exchangeable, if $T \circ g = g \circ T$ for any $g \in \mathcal{S}_{\infty}$.

Throughout the paper we use  the following notations. We denote by $\mathcal{P}(X)
$ the  space of Borel probability measures on topological space $X$,
by $\mathcal{P}_{ex}(\mathbb{R}^{\infty})$ the space of exchangeable 
probability measures  on $\mathbb{R}^{\infty}$, and by
$
\mathcal{P}_2(\mathbb{R})
$ the space of Borel probability measures on $\mathbb{R}$ with finite second moments.
We use notation $W_2(P,Q)$ for the standard quadratic Kantorovich distance  between measures $P,Q$
on some metric space.

We are interested in the following transportation problems.
\begin{problem} \label{ekp} {\bf Exchangeable Kantorovich problem.}
Given $\mu, \nu \in \mathcal{P}_{ex}(\mathbb{R}^{\infty})$
find the minimum $K(\pi)$ of the functional
$$
\pi \mapsto  \int (x_1-y_1)^2 \ d\pi
$$
on the set $\mathcal{P}_{ex}(\mu,\nu)$ of exchangeable measures on $\mathbb{R}^{\infty} \times \mathbb{R}^{\infty}$
with marginals $\mu,\nu$.
\end{problem}

\begin{problem} \label{emp}{\bf Exchangeable Monge problem.}
Given $\mu, \nu \in \mathcal{P}_{ex}(\mathbb{R}^{\infty})$ find a Borel exchangeable mapping $T: \mathbb{R}^{\infty} \mapsto \mathbb{R}^{\infty}$ 
such that the measure
$$
\pi = \mu \circ (x, T(x))^{-1}
$$
is a solution to the exchangeable Kantorovich problem.

The mapping $T$ is called exchangeable optimal transportation.
\end{problem}

The motivation for the study of these problems comes from the fact 
that the similar problems on $\mathbb{R}^n$  are
equivalent to the standard Monge and Kantorovich problems with the same marginals and the cost
function 
$\sum_{i=1}^n (x_i - y_i)^2$ (see \cite{KZ}, \cite{Moameni}, \cite{Zaev1}).
Thus the problems (\ref{ekp}), (\ref{emp}) can be viewed as natural generalizations
of the standard Monge-Kantorovich problem to the case of infinite-dimensional
exchangeable marginals. Note that two different infinite-dimensional exchangeable marginals on $\mathbb{R}^{\infty}$
have infinite Kantorovich distance  if one defines it in the standard way (via minimization of $\int (x-y)^2_{l_2} d \pi$).  In contrast to this, the value of the corresponding minimum of the Kantorovich potential 
is a squared distance on the space $\mathcal{P}_{ex}(\mathbb{R}^{\infty})$.
More explanations and results  can be found in \cite{KZ}, \cite{Zaev2}. See also \cite{Vershik} for similar problems 
on graphs.

It will be assumed throughout that
\begin{equation}
\label{integr}
\int x_1^2 d \mu + \int y_1^2 d \nu < \infty.
\end{equation}
Since the cost function is continuous, the solvability of the Kantorovich problem can be shown by the standard compactness arguments.

The paper is organized as follows. In Section 2 we show that the exchangeable Monge
problem is equivalent to the classical Monge problem on a convex subset of the Hilbert space $l^2$.
This is shown with the help of the de Finetti-type (ergodic) decomposition for transportations plans. The reduction to $l^2$
makes possible to apply the standard machinery of the transportation theory (duality, convex analysis etc.)
to the existence problem. 

 In Section 3 we pursue a completely different approach, namely, we study when the  optimal transportation 
 is a limit of natural finite-dimensional approximations. We emphasize that this problem is far from being trivial.
 The affirmative answer is established under quite special assumptions on the marginals.
 Moreover, it implies an unexpected result on the structure of exchangeable measures with additional analytical properties. 
 To be precise: we approximate the marginals by their finite-dimensional projections $\mu_n, \nu_n$. We 
 prove that the solutions $T_n$ to the standard Monge problem for $\mu_n,\nu_n$ do converge 
 $\mu$-a.e. to the desired mapping $T$ provided $T_n$ are uniformly globally Lipschitz:
 $$
 \| T_n(x) - T_n(y)\| \le K \|x-y\|.
 $$
 These assumption can be verified for certain measures,
 in particular, in the following model situation:
 $\mu$ is the standard Gaussian  measure and $\nu$  is uniformly log-concave.
 Compare this to existence result of the Section 2 we get the following corollary: every exchangeable uniformly log-concave measure is a countable power of a 
one-dimensional distribution.

\section{Reduction to the Hilbert space}

Given a Borel probability measure $m$ on $\mathbb{R}$ we denote by $m^{\infty}$
its  countable power   (i.i.d. distributions with law $m$), which is a 
probability measure on $\mathbb{R}^{\infty}$.

According to a wide generalization of the classical de Finetti theorem (see \cite{Bo}, \cite{Kallenberg}) the exchangeable measures are precisely
the mixtures of the countable powers.

\begin{theorem} {\bf (generalized De Finetti theorem).}
For every Borel exchangeable measure $\mu$ on $\mathbb{R}^{\infty}$
there exists a Borel probability measure $\Pi$ on $\mathcal{P}(\mathcal{P}(\mathbb{R}))$ such that
$$
\mu(B) = \int m^{\infty}(B) \Pi(dm),
$$
for every Borel $B \subset \mathbb{R}^{\infty}$.
\end{theorem}

Next, given the de Finetti decomposition
of two marginals we apply the following decomposition theorem, which is a particular case of a result 
 from \cite{Zaev2} on ergodic decompositions of optimal transportation plans.

\begin{theorem}
\label{plandecom}{\bf \cite{Zaev2}.}
Assume we are given the 
de Finetti  decompositions
\begin{equation}
\label{munudecomp}
\mu  = \int_{X}  \mu^{\infty}_x \ d {\sigma_{\mu}}, \  \nu  = \int_{Y}  \nu^{\infty}_y \ d {\sigma_{\nu}}
\end{equation}
of the measures $\mu,\nu$, where $X=Y = \mathcal{P}(\mathbb{R})$
 and, similarly, the ergodic decomposition  of $\pi$:
\begin{equation}
\label{pidecomp}
\pi = \int_{\mathcal{P}(\mathbb{R}^2) } \pi_{x,y} d \delta.
\end{equation}
Then for  $\delta$-almost all $(x,y)$ the measure $\pi_{x,y}$ solves  the one-dimensional quadratic Kantorovich problem with  marginals
$\mu_x, \nu_y$:
$$
\int (t-s)^2 d \pi_{x,y}(t,s) = W^2 _2(\mu_x,\nu_y) =  \min_{\theta \in \mathcal{P}(\mu_x,\nu_y)} \int (t-s)^2 d \theta(t,s)
$$
and the following representation formula holds:
$$
 \min_{\pi \in \Pi_{ex}(\mu,\nu)} \int (x_1-y_1)^2 \ d\pi  = \inf_{\delta \in \Pi(\sigma_{\mu}, \sigma_{\nu})}
\int  
W^2_2(\mu^{x}, \nu^{y} ) \ d \delta. 
$$
\end{theorem}

It  follows immediately from Theorem \ref{plandecom} that the Monge problem
can be similarly decomposed in two Monge problems:
\begin{itemize}
\item [1)] Monge problem for measures
$\sigma_{\mu}, \sigma_{\nu}$
and the cost function $(\mu,\nu) \mapsto W^2_2(\mu,\nu)$
on $\mathcal{P}(\mathbb{R})$.

\item [2)] One-dimensional Monge problem for measures $\mu_x, \nu_y$ and the quadratic cost function.
\end{itemize}
The following conclusion is straightforward.

\begin{corollary}
\label{not-exists}
The exchangeable Monge problem admits solution if and only if problem 1) is solvable and, moreover, 
problem 2) is solvable for $\sigma_{\mu}$-almost all $\mu^x$ and $\sigma_{\nu}$-almost all $\nu^y$.

The exchangeable Monge problem is not always solvable. For instance, if $\mu$ is a countable power, but $\nu$ is not,
then there is no optimal transportation pushing forward  $\mu$ onto $\nu$.
\end{corollary}
\begin{proof}
All the statement are immediate except of the ''only if'' part.
We have to show that every exchangeable optimal transportation $T$
induces an optimal transportation mapping $\mathcal{T} \colon \mathcal{P}(\mathbb{R}) \mapsto  \mathcal{P}(\mathbb{R})$.
This follows easily from the fact that every exchangeable mapping $T$ is diagonal
almost everywhere (i.e. has the form $T(x) = (t(x_1), \ldots, t(x_n), \ldots)$ for some $t \colon \mathbb{R} \mapsto \mathbb{R}$)
with respect to any countable power $\mu^{\infty}_x$.
This is an immediate consequence of the fact that every exchangeable function $f$ is constant for $\mu^{\infty}_x$-almost all points
by the Hewitt-Savage $0-1$ law. Thus induced mapping can be defined as follows: $\mathcal{T}(\mu_x) = \mu_x \circ t^{-1}$.
The optimality of the latter mapping follows from Theorem \ref{plandecom}.
\end{proof}

Since the one-dimensional Monge problem admits a precise solution under appropriate easy-to-check sufficient conditions, the exchangeable Monge problem
is reduced to the Monge problem on the metric space
$$
(\mathcal{P}_2(\mathbb{R}), W_2(\mathbb{R}))
$$
with the cost function $W^2_2$.

Remarkably, the problem can be further reduced to 
a problem on a {\bf linear} space.
This  can be made with the help of the well know fact
that $(\mathcal{P}_2, W_2)$ is isomorphic to  a convex subset of $L_2([0,1])$.
The distance preserving isomorphism $$\mathcal{I} \colon \mathcal{P}_2 \mapsto L^2([0,1])$$ has the form
$$
\mathcal{I}(\mu) = F^{-1}_{\mu},
 $$ 
 where $F^{-1}_{\mu}$ is the inverse distribution function of $\mu$.  In case when the distribution function $F_{\mu}$ is 
 not  one-to-one we simply define
$$
F^{-1}_{\mu}(t) = \inf \{ s \colon \mu(-\infty,s] > t\}.
$$
Thus the set 
$$\mathcal{K} = \mathcal{I}(\mathcal{P}_2(\mathbb{R}))
$$ 
consists 
of non-decreasing right continuous mappings which belong to $L^2[0,1]$. 

After all we conclude that the exchangeable Kantorovich and Monge problems
are reduced to the same problems on the subset $\mathcal{K}$ of $l^2$ equipped with the standard 
$l_2$-metric.

The existence of optimal transportation mappings on the Hilbert space is known under assumptions given below. 
It was obtained in  \cite{CM} and can be constructed with the help of  by now standard arguments. Indeed, one 
can consider the solution $(\varphi, \psi)$ to the dual Kantorovich problem
\begin{equation}
\label{dualK}
 \int \varphi d \mu + \int \psi d \nu \to \sup, \ \varphi(x) + \psi(y) \le |x-y|^2_{l^2}.
\end{equation}
It follows from the general results on 
the dual Kantorovich problem that for every solution $\pi$ to the primal Kantorovich problem
there exists a solution $(\varphi,\psi)$ to (\ref{dualK}) such that
$$
\varphi(x) + \psi(y) \le  |x-y|^2_{l^2}
$$
and $\varphi(x) + \psi(y) = |x-y|^2_{l^2}$ $\pi$-a.e. From these relations we infer that for $\pi$-a.e. 
points $(x_0,y_0)$ one has $y_0 \in \partial \varphi(x_0)$, where $\varphi(x_0)$ is the  superdifferential
of $\varphi$ at $x_0$. To construct the corresponding optimal transportation (and prove uniqueness of solutions to
all the associated optimal transportation problems) 
 it is sufficient to 
ensure that $\partial \varphi(x_0)$ contains a unique element $\mu$-a.e. It was verified
\cite{CM} under assumption of regularity of $\mu$.

\begin{definition} \label{reg}
Assume that we are given a sequence of vectors $\{e_i\}$ such that the closure of $\mbox{\rm}{span}(\{e_i\})$ contains the topological support of $\mu$.
Disintegrate $\mu$  with respect to  $e_i$:
$$
\mu = \int_{X^{\bot}_i} \mu^{x} d \mu_i, \  \ \ \mu_i = \mu \circ Pr_i^{-1},
$$
where  $Pr_i$ is the orthogonal  projection onto $X^{\bot}_i = \{x \colon x \bot e_i\}$
and $\{\mu^x\}$ is the corresponding family of conditional measures.

The measure $\mu$ on $l^2$ is called regular if for
$\mu_i$-almost every  $x$ the conditional measure $\mu^x$ is atomless.
\end{definition}

In sum, the following result holds.
\begin{theorem} \cite{CM}.
Let $\mu, \nu$ be Borel probability measures on 
$$(\mathcal{P}(\mathbb{R}), W^2_2(\mathbb{R})) \sim (\mathcal{K}, \| \cdot\|_{l^2}) \subset l^2.$$ Assume that 
$$
\int |x|^2_{l_2} d \mu + \int |y|^2_{l_2} d \nu < \infty
$$
and the source measure $\mu$ 
is regular in the sense of Definition \ref{reg}. Then there exists the unique solutions $\pi$,  $(\varphi,\psi)$ 
to the primal and the  dual Kantorovich problems and the unique solution to the Monge problem, which has the form
$$
T(x) = x - \partial \varphi(x).
$$
\end{theorem}


\section{Finite-dimensional approximations and log-concavity}

In this section we pursue completely different  approach to the existence for the Monge problem.
We construct the optimal mapping as a limit of finite-dimensional approximations.
In should be emphasized that it is usually hard to capture 
the decomposition structure given by the de Finetti theorem if the exchangeable measure is given as a limit of finite-dimensional 
approximations. This is the reason why the main result of this section  looks completely unrelated to the de Finetti decomposition
and abstract sufficient conditions obtained in the previous section.

The projection  $P_n \colon \mathbb{R}^{\infty} \mapsto \mathbb{R}^{\infty}$ onto the first $n$ coordinates will be denoted by $P_n $:
 $$P_n (x) = (x_1, x_2, \cdots, x_n ,0, 0, \cdots).$$
  Let us consider the projections $\mu_n = \mu \circ P^{-1}_n, 
\nu_n = \nu \circ P^{-1}_n$ of the marginals.

Clearly, the measures  $\mu_n,\nu_n$  are exchangeable as well (considered as measures on $\mathbb{R}^n$, i.e. invariant with respect to any permutation of the first $n$ coordinates). Let 
$\pi_n$ be the solution to
 the corresponding finite-dimensional exchangeable Monge-Kantorovich problem,
i.e. 
$$
\int (x_1 - y_1)^2 \ d m \to \inf
$$
where the infimum is taken among all of $2n$-dimensional exchangeable  measures with marginals $\mu_n,\nu_n$.
Equivalently, one can solve the standard Monge-Kantorovich problem with the cost function $\sum_{i=1}^{n} (x_i - y_i)^2$ instead.

Let $$T_n(x) = \nabla \Phi_n(x)$$
 be the corresponding optimal transportation mapping.

{\bf Assumption A.}
There exists $K>0$ such that the  potentials $\Phi_n$ do satisfy
$$
\Phi_n(a) - \Phi_n(b) - \langle \nabla \Phi_n(b), a-b \rangle \le  K \|a-b\|^2
$$
for all $n$, $a,b \in \mathbb{R}^n$.

Equivalently, the dual potentials $\Psi_n$ satisfy
$$
\Psi_n(a) - \Psi_n(b) - \langle \nabla \Psi_n(b), a-b \rangle \ge \frac{\|a-b\|^2}{K}.
$$

\begin{remark}
Clearly, assumption {\bf A} is equivalent to the requirement that every optimal mapping
$\mathbb{R}^n \ni x \mapsto \nabla \Phi_n(x)$ is $K$-Lipschitz: $$|\nabla \Phi_n(x) - \nabla \Phi_n(y)| \le K |x-y|$$
on $\mathbb{R}^n.$
\end{remark}

\begin{theorem} 
\label{permutation-exist}
Under assumptions $\bf{A}$  and (\ref{integr}) there exists a solution $\pi$ to  the problem (\ref{emp}).
\end{theorem}
\begin{proof}
Since the marginals of $\{\pi_n\}$ constitute tight sequences, the sequence $\{\pi_n\}$ of measures on $\mathbb{R}^{\infty} \times \mathbb{R}^{\infty}$ is tight itself. Hence  one can extract a weakly convergent subsequence
(denoted for brevity again by $\{\pi_n\}$) $\pi_n \to \pi$. 
Clearly, ${\pi}$ is exchangeable and has marginals $\mu, \nu$.
Let us show that $\pi$ is a solution  to the problem (\ref{ekp}). Indeed,
assuming the contrary, we get that there exists another exchangeable  measure
$\tilde{\pi}$ such that
$$
\int (x_1 -y_1)^2 d \tilde{\pi} < \int (x_1 - y_1)^2 d \pi.
$$ 
It follows from the weak convergence and (\ref{integr}) that
\begin{equation}
\label{costconv}
\int (x_1 - y_1)^2 d \pi =  \lim_n \int (x_1 -y_1)^2 d \pi_n.
\end{equation}
Hence $\int (x_1 - y_1)^2 d \tilde{\pi} <  \int  (x_1 -y_1)^2 d \pi_N $ for some $N$.
But this contradicts to optimality of $\pi_N$, because the projection of $\tilde{\pi}$ onto
$\mathbb{R}^N \times \mathbb{R}^N$ satisfies the constraints and gives a better value
to the Kantorovich functional.

 By the change of variables
$$ \int \partial_{x_i} \Phi_n^2 \ d \mu = \int \partial_{x_i} \Phi_n^2 \ d \mu_n  = \int y^2_i d \nu_n =  \int y^2_i d \nu < \infty$$
for every $i \le n$.
Let us pass to a subsequence of the sequence $\{\partial_{x_i} \Phi_n\}$
(denoted again by $\{\partial_{x_i} \Phi_n\}$). Applying the diagonal method
one can assume without loss of generality that
$$
\partial_{x_i} \Phi_n \to T_i
$$
weakly in $L^2(\mu)$ for every $i$.
We will show that $T = (T_1, T_2, \ldots, T_n, \ldots)$ is the desired mapping.
By standard measure-theoretical arguments it is sufficient to show that 
$
\partial_{x_i} \Phi_n \to T_i
$
in measure.

Consider the following quantity
$$
D_n = \int \bigl(\Phi_n(x) + \Psi_n(y) - \sum_{i=1}^n x_i y_i \bigr) d \pi,
$$
where $\Psi_n$ is the Legendre transform of $\Phi_n$ (the dual potential).
Since the integrand is nonnegative, one has $D \ge 0$.
Since  $\int \Phi_n d \pi = \int \Phi_n d \mu = \int \Phi_n d \mu_n = \int \Phi_n d \pi_n$, 
$\int \Psi_n d \pi = \int \Psi_n d \nu = \int \Psi_n d \nu_n = \int \Psi_n d \pi_n$,
and $\Phi_n + \Psi_n =  \sum_{i=1}^n x_i y_i $ $\pi_n$-almost everywhere, we get
$$
D_n =\int  \sum_{i=1}^n x_i y_i (d \pi_n - d \pi).
$$
Exchangeability of $\pi, \pi_n$ implies that  all the pairs $(x_i,y_i)$ are equally distributed. Hence
$$
D_n =n \int x_1 y_1 (d {\pi}_n - d \pi).
$$
We get, in particular, that 
\begin{equation}
\label{zeroconv}
\lim_{n \to\infty} \frac{D_n}{n}
=0.
\end{equation} 
This follows easily from the weak convergence $\pi_n \to \pi$ and (\ref{costconv}).

In the other hand
$$
D_n = \lim_m D_{n,m},$$
where
$$  D_{n,m} =  \int (\Phi_n(x) + \Psi_n(y) - \sum_{i=1}^n x_i y_i) d \pi_m
=  \int (\Phi_n(x) + \Psi_n(\nabla \Phi_m) - \sum_{i=1}^n x_i \partial_{x_i} \Phi_m) d \pi_m.
$$
Indeed, by the same arguments as above we show that $\int (\Phi_n(x) + \Psi_n(y)) d \pi_m = \int \Phi_n d \mu + \int \Psi_n d \nu$ ( $m  \ge n$)
and  $\int x_i y_i d \pi_m \to \int x_i y_i d \pi$ for every $i \le n$.

Taking into account the identity 
$$
\Phi_n(x) = - \Psi_n(\nabla \Phi_n) + \sum_{i=1}^n x_i \partial_{x_i} \Phi_n
$$
one obtains
$$
D_{n,m} = \int  \Psi_n(\nabla \Phi_m(x)) - \Psi_n(\nabla \Phi_n(x)) - \sum_{i=1}^n x_i (\partial_{x_i} \Phi_m - \partial_{x_i} \Phi_n) d \mu.
$$
Assumption {\bf A} implies
$$
D_{n,m} \ge \frac{1}{K} \int |Pr_n \nabla \Phi_m - \nabla \Phi_n|^2 d \mu.
$$
Passing to the limit $m \to \infty$ and applying the $L^2(\mu)$-weak convergence $\partial_{x_i} \Phi_m \to T_i$  one gets
by the well-known properties of the $L^2$-weak convergence
$$
K D_n \ge \int |Pr_n  T- \nabla \Phi_n|^2 d \mu.
$$
Since $Pr_n T$ and $\nabla \Phi_n$  commute with  permutations of the first $n$ coordinates, one gets
$$
 \frac{KD_n}{n} \ge  \int (T_1 - \partial_{x_1} \Phi_n)^2 d \mu. 
$$
Then (\ref{zeroconv}) implies $
\partial_{x_1} \Phi_n \to T_1
$
in measure. By the exchangeability  the same holds for every $x_i$: $\lim_n \partial_{x_i} \Phi_n = T_i$. The proof is complete.
\end{proof}

As an interesting  byproduct we get a characterization of the uniformly log-concave exchangeable measures.

We recall that a probability measure $\mu$  on $\mathbb{R}^n$ is called log-concave if it has the form $e^{-V} \cdot \mathcal{H}^{k}|_{L}$, where $\mathcal{H}^k$ is the $k$-dimensional Hausdorff
measure, $k \in \{0,1, \cdots, n\}$, $L$ is an affine subspace,  and $V$ is a convex function. 

In what follows we consider uniformly log-concave measures. Roughly speaking, these  are the measures with potential $V$
satisfying 
$$
V(x) - V(y) - \langle \nabla V(y), x-y \rangle \ge \frac{K}{2}|x-y|^2, \  \ \ K >0
$$ 
which is equivalent to $D^2 V \ge K \cdot \mbox{Id}$ in the smooth (finite-dimensional) case.

More precisely, we say that a probability measure $\mu$ is $K$-uniformly log-concave  ($K > 0$) if for any $\varepsilon>0$ the measure $\hat{\mu} = \frac{1}{Z} e^{\frac{K-\varepsilon}{2} |x|^2} \cdot \mu$ is  log-concave  for a suitable renormalization factor $Z$.
According to a classical result of C. Borell (\cite{Borell})  the projections  of log-concave measures are log-concave (this is in fact a corollary of the Brunn-Minkowski theorem). It can be easily checked that the uniform log-concavity is preserved by projections as well.
We can extend this notion to the infinite-dimensional case. Namely, we call a probability measure $\mu$ on a locally convex space $X$ log-concave ($K$-uniformly log-concave with $K>0$)  if its images  $\mu \circ l^{-1}$, $l \in X^*$ under linear continuous
functionals  are  all log-concave ($K$-uniformly log-concave with  $K>0$).

Another classical result we apply below is the famous  Cafarelli's contraction theorem. Here is the version from \cite{Kol2010}.
\begin{theorem} \label{Cafcon}{\bf (Caffarelli contraction theorem).}
Let $\nabla \Phi$ be the optimal transportation of the probability measure
$\mu = e^{-V} dx$ into $\nu = e^{-W} dx$. Assume that for some positive $c, C$ one has
$D^2 V \le C \cdot \rm{Id}$, $D^2 W \ge c \cdot \rm{Id}$. Then $\nabla \Phi$ is Lipschitz with 
$\| \nabla \Phi\|_{Lip} \le \sqrt{\frac{C}{c}}$.   
\end{theorem}

\begin{remark}
Clearly, Theorem \ref{Cafcon} provides a tool for verification of Assumption {\bf A}. The authors do not know other instruments to establish {\bf A} with comparable level of generality.
\end{remark}

\begin{remark}
As we already mentioned, the lower bound for the potential of measures is preserved under projections.
It is interesting that the upper bound for the potential 
\begin{equation}
\label{upbd2}
D^2 W \le K
\end{equation}
is  preserved under projections as well, i.e. all the projections of  $\nu = e^{-W} dx$ which satifies (\ref{upbd2})
have again the same property. For smooth potentials this can be checked by direct computations.
\end{remark}

\begin{theorem}
\label{exch-log-conc}
Every exchangeable uniformly log-concave measure $\nu$ is a countable power of a one-dimensional uniformly log-concave measure.
\end{theorem}
\begin{proof}
Theorem \ref{permutation-exist} implies existence of an exchangeable transportation mapping $T$ pushing forward the standard Gaussian measure
$\gamma = \gamma^{\infty}, \gamma(dx) = \frac{1}{\sqrt{2 \pi}} e^{-\frac{x^2}{2}}$ onto $\nu$. Indeed, Assumption {\bf A}
follows from the Caffarelli
contraction theorem and the fact that the finite-dimensional projections of $\nu$ are uniformly log-concave.  The result follows from Corollary \ref{not-exists}.
\end{proof}

\begin{remark}
The assumption of uniform log-concavity in Theorem \ref{exch-log-conc} is important and can not be replaced  by the weaker assumption of  log-concavity. There exist log-concave exchangeable measures which are not product measures.
For example, let $m = e^{-V(x)} dx$  be a one-dimensional log-concave  probability measure.
The measure $$\tilde{\nu} = \prod_{i=1}^{\infty} e^{-V(x_i + t)} dx_i \cdot \frac{1}{\sqrt{2 \pi}} e^{-\frac{t^2}{2}}$$ is a log-concave measure on $\mathbb{R}^{\infty} \times \mathbb{R}$.
Its projection on $x$-coordinates is log-concave by the result of C.~Borell and exchangeable, but not a product measure. 
\end{remark}

\end{document}